\documentclass[10pt]{article}
\usepackage{latexsym,amsmath,amsthm,amssymb,float}
\usepackage[sorted]{amsrefs}
\usepackage{graphicx}
\usepackage[colorlinks=true, linkcolor=blue]{hyperref}
\usepackage{geometry}
\usepackage{enumitem}
\usepackage{tikz}
\usepackage{mathtools}
\usepackage{thm-restate}
\usepackage{cleveref}
\usepackage{verbatim}
\usepackage{ragged2e}\RaggedRight

\newcommand{\yoa}[1]{{\bf\color{orange}{[Alexander: #1}]}}

\newtheorem{theorem}{Theorem}[section]

\newtheorem{lemma}[theorem]{Lemma}

\newtheorem{corollary}[theorem]{Corollary}
\newtheorem{question}[theorem]{Question}

\newtheorem{remark}[theorem]{Remark}

\newcommand{\thistheoremname}{}
\newtheorem*{genericthm*}{\thistheoremname}
\newenvironment{namedthm*}[1]
  {\renewcommand{\thistheoremname}{#1}%
   \begin{genericthm*}}
  {\end{genericthm*}}

\usepackage[color=red,textsize=footnotesize]{todonotes}

\setlist[enumerate]{topsep=0pt,partopsep=1ex,parsep=1ex}
\usepackage{dsfont}

\newcommand{\ceil}[1]{\left \lceil #1 \right \rceil}

\title{$3$-Neighbor bootstrap percolation on two-dimensional grids}

\author{Neal Bushaw\thanks{Dept. of Math. \& Appl. Math., Virginia Commonwealth University, Richmond, USA {\tt nobushaw@vcu.edu} } \and Alexander Clifton\thanks{Department of Theoretical Computer Science, Czech Technical University in Prague, Prague, Czechia {\tt alexander.clifton@fit.cvut.cz}. The author was supported in part by the Institute for Basic Science (IBS-R029-C1) and by the grant 23-06815M of the Grant Agency of the Czech Republic.}}

\begin{document}

\maketitle
\begin{abstract}
In the $3$-neighbor bootstrap percolation process, a vertex becomes (and remains) infected if at least three of its neighbors are infected. We say that an initial configuration of infected vertices percolates if eventually all vertices are infected. We exactly determine the size of the minimum percolating set for the $3$-neighbor bootstrap percolation process on all remaining open cases for rectangular grid graphs $P_m\square P_n$. This extends earlier work of Dukes, Noel, and Romer.
\end{abstract}
\section{Intro}
Bootstrap percolation is a graph-theoretic infection model whose origins lie in statistical physics. It is a simple model of the cellular automata first introduced by Von Neumann, spurred by a suggestion of Ulam.  We think of some vertices of a graph as being infected, and that this infection spreads over time whenever an uninfected vertex is a neighbor to many (at least $k$) infected vertices. This process has been well explored by numerous authors in diverse fields; for a survey with applications to several areas, see \cite{AdlerLev} and its many references.

To be precise in our combinatorial setting, we define the {\bf{$k$-bootstrap percolation}} process on a graph $G$ starting from an initially infected set of vertices $\mathcal{I}_0\subseteq V(G)$.  We then inductively define the set of infected vertices at each successive discrete time step, setting \[\mathcal{I}_{t+1}=\mathcal{I}_t\cup\{x\in V(G)-\mathcal{I}_t:|N(x)\cap\mathcal{I}_t|\ge k\}.\]

We are interested in the long term behavior of these initially infected sets. In particular, does the infection stop with some vertices remaining uninfected, or does it spread to eventually infect every vertex?  In the latter case, we say that $\mathcal{I}_0$ {\bf{percolates}}.  With this long term behavior in mind, we set $\left\langle\mathcal{I}_0\right\rangle=\bigcup_{t\ge 0}\mathcal{I}_t$, and say that $\mathcal{I}_0$ percolates when $\left\langle\mathcal{I}_0\right\rangle=V(G)$.

Due to bootstrap percolation's origins in statistical physics, a great deal of research has investigated the case where $\mathcal{I}_0$ is chosen randomly among the vertices of an $n$-dimensional grid -- for a quite general survey, see e.g., \cite{RM}. Here, we set our eyes on a very natural extremal question: given a graph $G$ and an integer $k$, what is the minimum possible cardinality of a percolating set under $k$-neighbor bootstrap on $G$? 

While this problem sounds innocent, determining the minimum is NP-hard, even in the case of $2$-neighbor percolation on graphs of bounded degree \cites{Centano, Ning, DreyerRoberts}. Determining the minimum size of a percolating set has been studied for many classes of graphs, including trees \cite{Centano, DreyerRoberts, Riedl}, chordal graphs \cites{Centano, Bessy, Chiang}, hexagonal grids \cite{Adams}, and others.  In this manuscript, we focus on rectangular grids.

For the $m\times n$ rectangular grid $P_m\square P_n$, we let $s_k(m,n)$ denote the minimum number of initial infected cells required to percolate under the $k$-neighbor bootstrap percolation process.

Determining $s_2(n,n)$, the minimum size of a 2-percolating set on the square grid, is a very well known special case which has often been given as a puzzle, dating back to at least the 1980s \cite{BB,Winkler,Times}.  There is an easy perimeter argument showing that $s_2(n,n)=n$; we use this idea in a considerably more general setting in Section \ref{sec:Rectangles}.

 Our main result is to fully determine $s_3(m,n)$ for all pairs $m,n$, where it was not previously known.  We will compare these to a version of a folklore general lower bound which depends on the number of even dimensions; for completeness, we provide a proof in Section \ref{sec:Rectangles}.  Rather than repeatedly writing out these cases, we will use the following shorthand:

\[
LB :=
\begin{cases*}
  \left\lceil \dfrac{mn + m + n}{3} \right\rceil
    & for $m,n$ odd,\\[4pt]
  \left\lceil \dfrac{mn + m + n + 2}{3} \right\rceil
    & for exactly one of $m,n$ even,\\[4pt]
  \left\lceil \dfrac{mn + m + n + 4}{3} \right\rceil
    & for both $m,n$ even.
\end{cases*}
\]

We summarize values of $s_3(m,n)$ that were previously known. Because $s_3(m,n)=s_3(n,m)$, some values are included implicitly.
\begin{theorem}\label{thm:oldcases}
    $s_3(m,n)=LB$ when $m\equiv 2\pmod{3}$ or $m\equiv n\equiv 0\pmod{6}$. Also, $s_3(1,n)=n, s_3(2,n)=n+2$ for $n\ge 2$, and $s_3(3,n)=\lceil{\frac{3n+1}{2}\rceil}$ for $n\ge 3$.
\end{theorem}



The results for $m\equiv 2\pmod{3}$ and $m\equiv n\equiv 0\pmod{6}$ follow from~\cite{DNR23} and the result for $m=3$ is due to Hed\v{z}et and Henning~\cite{HH23}.  For $m=2$ see, e.g., \cite{DNR23}.

Building on work of Benevides, Bermond, Lesfari, and Nisse~\cite{BBL24} who studied the case of $m=n$, Dukes, Noel, and Romer showed the following in \cite[Theorem 1.8]{DNR23}:

\begin{theorem}\label{rectangleC0}
    The value $s_3(m,n)=\frac{mn+m+n}{3}$ if and only if $m=n=2^k-1$ for some integer $k$.
\end{theorem}

\begin{corollary}\label{cor:rect1}
   Using \cite[Lemma 4.6]{DNR23}, it follows that $s_3(m,n)=\frac{mn+m+n}{3}+1$ when $m\equiv n\equiv 1\pmod{6}$ but either $m\neq n$ or $m+1$ is not a power of $2$.
\end{corollary}

Before stating our general result, we resolve a few small cases:
\begin{restatable}[]{theorem}{rectanglesmall}\label{thm:rectsmall}
For $n\ge 4$,
        \begin{align*}
    s_3(4,n)&=\begin{cases*}
    $LB$& when $n\equiv{2}\pmod{3}$, $n$ even, or $n=7$,\\
    $LB+1$& otherwise.
    \end{cases*}\\
     s_3(6,n)&=\begin{cases*}
    $LB$& when $n\equiv{2}\pmod{3}$, $n$ even, or $n=7, 9, 15$,\\
    $LB+1$& otherwise.
    \end{cases*}
        \end{align*}
\end{restatable}

For larger $m,n$, the following holds.

    \begin{restatable}[]{theorem}{rectangle}\label{thm:rectall}
        For $m,n\ge 7$,
        \[
        s_3(m,n)=LB,
        \]
        except: $s_3(m,n)=LB+1$ when $m,n\equiv 1\pmod{6}$ and $m,n\equiv 3\pmod{6}$ (unless $m=n=2^k-1$ for some integer $k$, where the value is determined by Theorem \ref{rectangleC0}).
    \end{restatable}


A note on notation: as it is convenient to us, we will use a mixture of `grid' terminology and graph terminology. We will refer to an infected agent as either a `vertex' or `cell', yet still talk about its neighbors and its degree.  Whenever this language introduces ambiguity, we will be careful to clarify what is meant through precise definitions and/or diagrams. Further, we stick to index standards where possible -- a vertex indexed by $(i,j)$ lies in row $i$ and column $j$, where we view column $1$ as leftmost and row $1$ as top.

Throughout, we shall refer to the {\bf{perimeter}} of a set of infected vertices $S$, denoted $P(S)$; this is simply the number of edges between infected vertices and uninfected vertices in the host graph. To be precise, we set $P(S)=|E(S,\overline{S})|$ for any $S\subseteq V(G)$.

Throughout, we will make use of the observations encapsulated in the following lemma.
\begin{lemma}[Lemma 3.1 in \cite{DNR23}]\label{lem:trees}
    Let $k\ge 1$, and consider a graph $G$ with maximum degree $\Delta(G)\le k+1$ and $S\subseteq V(G)$.  The set $S$ percolates in the $k$-neighbor bootstrap process on $G$ if and only if $S\supseteq\{v\in V(G):d(v)<k\}$ and every component of $G-S$ is a tree with at most one vertex having degree $k$ in $G$.
\end{lemma}
In particular, in the setting of rectangular grids, every corner must be initially infected and there is no path of uninfected cells whose endpoints both lie on the boundary. Further, every row, column, and $2\times 2$ sub-grid contains at least one infected cell.

\section{Rectangles}\label{sec:Rectangles}
Recall that $LB=\lceil{\frac{mn+m+n}{3}\rceil}$ for $m,n$ odd, $\lceil{\frac{mn+m+n+2}{3}\rceil}$ if one is even, and $\lceil{\frac{mn+m+n+4}{3}\rceil}$ if both are even. In this section, we will prove the following:
\rectangle*

First, we establish that LB is a lower bound for $s_3(m,n)$ for all pairs $m,n$.

\begin{theorem}\label{thm:rectlower}
$s_3(m,n)\ge LB$.
\end{theorem}
\begin{proof}
A corner vertex has degree two, so it must be initially infected in order for the set to percolate. Similarly, if two adjacent vertices on the boundary are uninfected, they each have only two neighbors besides each other, and can never become infected. As a result, for a set of infected vertices to percolate, there are never two consecutive uninfected vertices on the boundary. Thus, when $m$ is even, a side of length $m$ must have at least $m/2+1$ initially infected vertices, and consequently has at least one pair of adjacent infected vertices. Thus, there are at least $2a$ pairs of adjacent infected vertices where $a\in\{0,1,2\}$ is the number of $m,n$ which are even.

Recall that the perimeter of a set of infected vertices $S$ is the number of edges between infected and uninfected vertices. When a vertex of $P_m\square P_n$ becomes infected, at least three such edges become edges between two infected vertices while at most one edge is newly between an infected and uninfected vertex. Thus, $P(S)$ decreases by at least $2$ when a new vertex is infected and must decrease by $3$ when a boundary vertex is infected. Thus for the initial set of infected vertices $S_0$ with $b$ uninfected boundary vertices, we have that $P(S_0)-2(mn-|S_0|)-b\ge 0$ or $P(S_0)\ge 2mn-2|S_0|+b$.

To bound $|S_0|$, note that $|S_0|$ necessarily contains $4$ degree $2$ vertices. If $b$ boundary vertices are uninfected, that is missing from $S_0$, then $2m+2n-8-b$ degree $3$ vertices are infected, meaning $|S_0|-4-(2m+2n-8-b)=|S_0|-2m-2n+b+4$ degree $4$ vertices are in $|S_0|$.  Tabulating degrees, this gives an upper bound of \[P(S_0)\le 2(4)+3(2m+2n-8-b)+4(|S_0|-2m-2n+b+4)=4|S_0|-2m-2n+b,\] but we further note that there are at least $2a$ pairs of adjacent infected vertices on the boundary, with each pair reducing this upper bound by $2$ to yield $P(S_0)\le 4|S_0|-2m-2n+b-4a $. Thus,
\begin{align*}
    4|S_0|-2m-2n+b-4a&\ge 2mn-2|S_0|+b\\
    6|S_0|&\ge 2mn+2m+2n+4a\\
    |S_0|&\ge\frac{mn+m+n+2a}{3}.
\end{align*}

Plugging in each possible value of $a$ and recognizing that $s_3(m,n)$ is always an integer gives the desired bound in each parity case.
\end{proof}

\begin{remark}\label{WhenBoundTight}
Note that if there are initially two adjacent infected cells or an uninfected cell adjacent to four infected cells, then the same proof gives $s_3(m,n)\ge \frac{mn+m+n+2a+1}{3}$.
\end{remark}

\begin{corollary}
    If $m,n$ odd but either $m\neq n$ or $m+1$ is not a power of $2$, then $s_3(m,n)\ge \lceil{\frac{mn+m+n+1}{3}\rceil}$.
\end{corollary}
\begin{proof}
    This follows directly from Theorem~\ref{thm:rectlower} except when $\frac{mn+m+n}{3}$ is an integer. In that case, the claim follows from Theorem~\ref{rectangleC0}.
\end{proof}

We now prove Theorem~\ref{thm:rectsmall}; this is split into the cases $m=4$ and $m=6$. 

\begin{theorem}\label{width4}
     For $n\ge 4$, we have $s_3(4,n)=\lfloor{\frac{5n+5}{3}\rfloor}+1$ when $n=7$ or $n\equiv 2\pmod{3}$ and $s_3(4,n)=\lfloor{\frac{5n+5}{3}\rfloor}+2$ otherwise.
\end{theorem}
\begin{proof} A general upper bound of $\lfloor{\frac{5n+5}{3}\rfloor}+2$ as well as the case of $n=7$ are found in Theorem 2.3 of \cite{HH23}. Furthermore, if $n\equiv 2\pmod{3}$, Theorem \ref{thm:oldcases} gives an answer of $\frac{4n+n+4+4}{3}=\frac{5n+8}{3}=\lfloor{\frac{5n+8}{3}\rfloor}$. It remains to show a lower bound of $\lfloor{\frac{5n+5}{3}\rfloor+2}$ for the remaining values of $n$.

When $n\equiv 0,4\pmod{6}$, \[
LB=\lceil{\frac{4n+n+4+4}{3}\rceil}=\lceil{\frac{5n+2}{3}\rceil}+2=\lfloor{\frac{5n+5}{3}\rfloor}+2
\]
where the last equality follows from the fact that $\frac{5n+5}{3}$ is not an integer.

When $n\equiv 1,3\pmod{6}$, we revisit the proof of Theorem~\ref{thm:rectlower}. If there are at least $c$ pairs of adjacent infected vertices, regardless of location, the resulting bound is instead $|S_0|\ge\frac{mn+m+n+c}{3}=\frac{5n+4+c}{3}$. 

When $n\equiv 1\pmod{6}$, then $c\ge 4$ is enough to give a lower bound of $\lceil{\frac{5n+8}{3}\rceil}=\lfloor{\frac{5n+5}{3}\rfloor}+2$, while for $n\equiv 3\pmod{6}$, then $c=3$ is enough to give a lower bound of $\lceil{\frac{5n+7}{3}\rceil}=\lfloor{\frac{5n+5}{3}\rfloor}+2$. Note that in the first and last columns, which have length four, there is necessarily at least one pair of adjacent infected cells. For $n\equiv 1\pmod{6}, n>7$ we must show that there are at least two more such pairs, while for $n\equiv 3\pmod{6}, n>3$, we must show there is at least one more such pair. First note that in order to percolate, all corners must be initially infected. If either the first or last row, which has odd length, contains some pair of adjacent infected cells, it would contain a second such pair, so we may assume that neither the first or last row has a pair of adjacent infected cells. It is now sufficient to show that the first four columns and last four columns each necessarily contain another pair of adjacent infected cells. As $n>7$ for the case $n\equiv 1\pmod{6}$, these will necessarily be distinct pairs.

We focus our attention on the first four columns, with an identical argument possible for the last four columns. Assume for the sake of contradiction that aside from one pair in the first column, there are no pairs of adjacent infected cells. By earlier reasoning, rows $1, 4$ have infected cells in columns $1,3$ and not in columns $2,4$. Without loss of generality, $(2,1)$ is uninfected while $(3,1)$ is infected.
In order to percolate, there is no path between two boundary cells consisting entirely of uninfected cells, so by virtue of $(1,2)$ and $(2,1)$ being uninfected, we have that $(2,2)$ must be infected. 
Assuming for the sake of contradiction that there are no further pairs of adjacent infected cells, we get that $(2,3)$, $(3,2)$, and $(3,3)$ are uninfected.
Now, to avoid a path of uninfected cells from $(4,2)$ to $(4,4)$, we have that $(3,4)$ is infected, meaning $(2,4)$ is uninfected. This creates a path of uninfected cells from $(4,2)$ to $(1,4)$, yielding a contradiction. Thus, we may conclude there is a second pair of adjacent infected cells within the first four columns, as desired.
\end{proof}

\begin{theorem}\label{width6}
For $n\ge6$, $s_3(6,n)=\lceil{\frac{7n+10}{3}\rceil}$ when $n$ even, $s_3(6,n)=\lceil{\frac{7n+8}{3}\rceil}$ when $n\equiv 2\pmod{3}$ or $n=7,9,15$, and $s_3(6,n)=\lceil{\frac{7n+11}{3}\rceil}$ otherwise.
\end{theorem}
\begin{proof}
    The cases $n\equiv 0,2,5\pmod{6}$ are handled in Theorem~\ref{thm:oldcases}. For $n\equiv 4\pmod{6}$, $s_3(6,n)\ge LB=\lceil{\frac{7n+10}{3}\rceil}$. For $n\equiv 1,3\pmod{6}$,  $s_3(6,n)\ge LB=\lceil{\frac{7n+8}{3}\rceil}$. For $n=7,9,15$, we present the following matching upper bound constructions:

    For $6\times 7$, infect the cells in odd columns for rows $1,3,6$, cells in columns $2,6$ for row $2$, cells in columns $1,7$ for row $4$, and cells in even columns for row $5$.

For $6\times 9$, infect the cells in odd columns for rows $1,3,6$, cells in columns $2,6,8$ for rows $2,5$, and cells in columns $1,4,9$ for row $4$.

For $6\times 15$, infect the cells in odd columns for rows $1,3,6$, cells in columns $2,6,10,14$ for row $2$, cells in columns $1,4,8,12,15$ for row 4, and cells in columns $2,6,8,10,14$ for row 5.

For the remaining cases with $n\equiv 3\pmod{6}$, the upper bound of $\ceil{\frac{7n+11}{3}}=LB+1$ follows from~\cite[Lemma 4.5]{DNR23}. For the remaining cases with $n\equiv 1\pmod{3}$, we get an upper bound of $LB=\ceil{\frac{7n+10}{3}}$ for $n\equiv 4\pmod{6}$ and $LB+1=\ceil{\frac{7n+11}{3}}$ for $n\equiv 1\pmod{6}$ from the following constructions (which differ only in their last three columns), as illustrated in Figure \ref{fig:6by1}.

\begin{figure}[H]
    \centering
    \includegraphics[width=\textwidth]{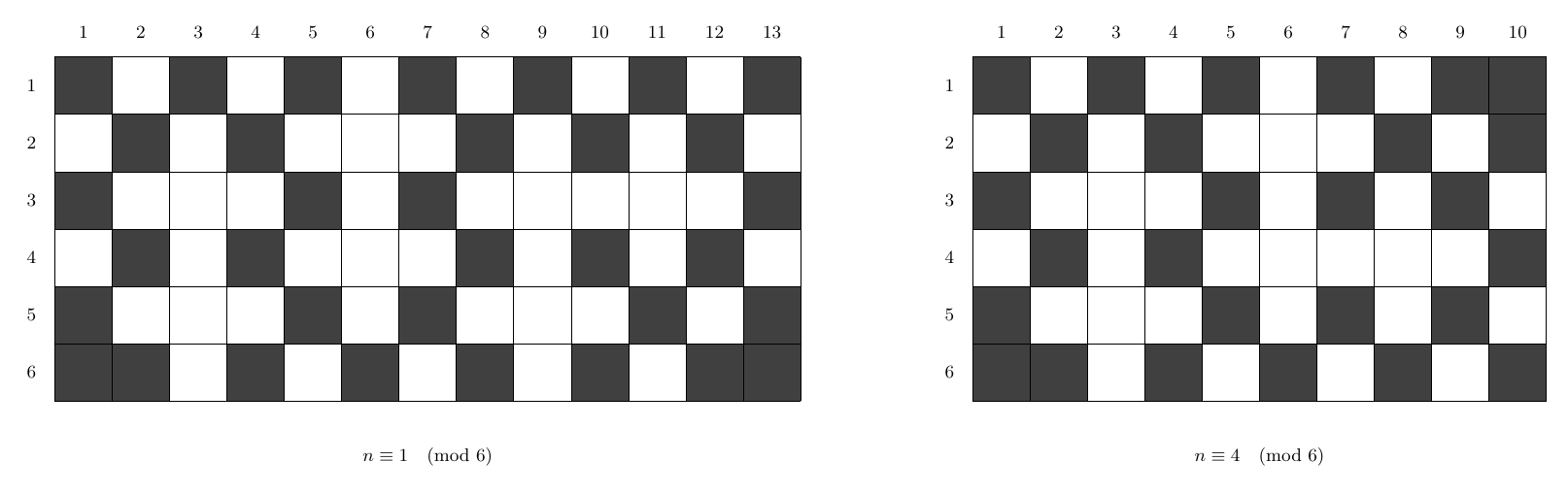}
    \caption{Minimum percolating sets in the \mbox{$6\times13$} and \mbox{$6\times10$} grids, as per Theorem \ref{width6}.}\label{fig:6by1}
\end{figure}

For the first $n-3$ columns, in columns $1,5\pmod{6}$, infect the cells in the odd rows; in columns $2,4\pmod{6}$, infect the cells in the even rows; in columns $3\pmod{6}$, infect only the first cell; and in columns $0\pmod{6}$, infect only the last cell. Additionally infect $(6,1)$. When $n\equiv1\pmod{6}$, the infected cells in the last three columns are $(1,n-2),(5,n-2),(2,n-1),(4,n-1),(6,n-1),(1,n),(3,n),(5,n)$, and $(6,n)$. When $n\equiv 4\pmod{6}$, the infected cells in the last three columns are $(2,n-2),(6,n-2),(1,n-1),(3,n-1),(5,n-1),(1,n),(2,n),(4,n)$, and $(6,n)$. In either case $\pmod{6}$, the total number of infected cells is $1+7(\frac{n-4}{3})+3+9=\frac{7n+11}{3}$, as desired and it suffices to check that every cell becomes infected.

Within the first $n-5$ columns, any initially uninfected cells in columns $1,2,4,5\pmod{6}$ become infected immediately, as do the initially uninfected cells in columns $n-1,n$. This in turn infects the remaining uninfected cells in the first $n-5$ columns. It is then easy to verify that any uninfected cells in the remaining three columns eventually become infected. 

To prove the lower bound of $LB+1$ in the remaining cases, we revisit the proof of Theorem~\ref{thm:rectlower} to note that if there are at least $c$ pairs of adjacent infected vertices, then $|S_0|\ge\frac{7n+6+c}{3}$. However, there is a further way to increase the lower bound, by considering cells that already have four infected neighbors by the time they become infected. These include all the initially uninfected cells with four initially infected neighbors. Let $d$ be the number of initially uninfected cells with four initially infected neighbors. Whenever such a cell becomes infected, $P(S_0)$ decreases by $4$ instead of $2$, so we get in general, the inequality $P(S_0)-2(mn-|S_0|)-b-2d\ge 0$. Ultimately, this leads to a lower bound of $s_3(6,n)\ge\frac{7n+6+c+d}{3}$.

When $n\equiv 1\pmod{6}$, $c+d\ge 3$ is enough to give a lower bound of $\lceil{\frac{7n+9}{3}\rceil}=\lceil{\frac{7n+11}{3}\rceil}$. When $n\equiv 3\pmod{6}$, $c+d\ge 4$ is enough to give a lower bound of $\lceil{\frac{7n+10}{3}\rceil}=\lceil{\frac{7n+11}{3}\rceil}$. In the first and last columns, which have length six, there is necessarily at least one pair of adjacent infected cells. For $n\equiv 1\pmod{6}$, $n>7$, we must show that there is at least one more contribution to $c+d$ (either a pair of adjacent infected cells, or an uninfected cell surrounded on all four sides by infected cells), while for $n\equiv 3\pmod{6}, n>15$, we must show there are at least two more contributions to $c+d$, either by finding another pair of adjacent infected boundary cells or another cell which becomes infected by all four of its neighbors. First note that in order for all cells to become infected, all cells with degree less than $3$, i.e., all corners must be initially infected. If either the first or last row, which has odd length, contains some pair of adjacent infected cells, it would contain a second such pair, so we may assume that neither the first or last row has a pair of adjacent infected cells. It is now sufficient to show that the first eight columns and last eight columns each necessarily contain another pair of adjacent infected cells or an uninfected cell surrounded by four infected cells. For $n\equiv 3\pmod{6}$, we have $n\ge 16$, so these will necessarily be distinct occurrences in the cases where two such occurrences are necessary.

We focus our attention on the first eight columns, with an identical argument possible for the last eight columns. Assume for the sake of contradiction that aside from one pair in the first column, there are no pairs of adjacent infected cells. By earlier reasoning, rows $1,6$ have infected cells in columns $1,3,5,7$ but not columns $2,4,6,8$. Without loss of generality, the pair of adjacent infected cells in the first column is either in rows $1,2$ or rows $3,4$.

First suppose the pair of adjacent infected cells in the first column is in rows $1,2$. Thus $(3,1)$ and $(5,1)$ are uninfected. In order to percolate, there are no two consecutive uninfected cells along the boundary, so $(4,1)$ is infected. Avoiding further pairs of adjacent infected cells, we get that $(2,2), (2,3), (4,2)$, and $(5,3)$ are uninfected. In order to percolate, there is no path between two boundary cells consisting entirely of uninfected cells, so $(2,4), (3,2)$, and $(5,2)$ are infected. In turn, this means $(3,3)$ and $(3,4)$ are uninfected. To avoid an uninfected path from $(1,2)$ to $(6,4)$, at least one of $(4,3)$ and $(5,4)$ is infected, and at least one of $(4,4)$ and $(5,4)$ is infected. Since $(3,3)$ and $(3,4)$ are adjacent, they cannot both be infected, necessitating that $(5,4)$ is infected. Consequently, $(4,4)$ is uninfected. To prevent the $2\times 2$ square in rows $3,4$ and columns $3,4$ from being totally uninfected, it is necessary that $(3,4)$ is infected, a contradiction.

Now suppose that the pair of adjacent infected cells in the first column is in rows $3,4$. Assuming no more pairs of adjacent infected cells, we have that cells in columns $1,3,5,7$ of rows $2,5$ are uninfected and that $(3,2)$ and $(4,2)$ are uninfected. To avoid uninfected paths from $(1,2)$ to $(2,1)$ and from $(5,1)$ to $(6,2)$, we must have that $(2,2)$ and $(5,2)$ are infected.

To prevent the $2\times 2$ square in rows $3,4$ and columns $2,3$ from being totally uninfected, either $(3,3)$ or $(4,3)$ is infected. The partial infection pattern determined so far is horizontally symmetric, so without loss of generality, we may assume that $(3,3)$ is infected, which in turn means $(3,4)$ and $(4,3)$ are both uninfected. To prevent $(2,3)$ from being an uninfected cell surrounded by four infected cells, $(2,4)$ must be uninfected. Thus, to avoid a $2\times 2$ square of uninfected cells in rows $2,3$ and columns $4,5$, we have that $(3,5)$ is infected. Thus, $(4,5)$ and $(3,6)$ are uninfected. To avoid an uninfected path from $(1,4)$ to $(1,6)$, we have that $(2,6)$ is infected.

The adjacent cells $(4,4)$ and $(5,4)$ cannot both be infected. Regardless of which is uninfected, there is a path of uninfected cells from either $(1,4)$ or $(6,4)$ to $(5,5)$. To avoid extending this path all the way to the boundary cell in $(6,6)$, we have that $(5,6)$ is necessarily infected which forces $(4,6)$ to be uninfected.

Given the previous assignments which were required to avoid producing another pair of adjacent infected cells or an uninfected cell surrounded by four infected cells, we will show that one or the other occurs within columns $6,7,8$. To prevent a totally uninfected $2\times 2$ square in rows $3,4$ and columns $6,7$, either $(3,7)$ or $(4,7)$ is infected. Since the partial infection pattern determined so far for columns $6,7,8$ is horizontally symmetric, we may assume without loss of generality that $(3,7)$ is infected, meaning $(4,7)$ and $(3,8)$ are uninfected.

Now to avoid an uninfected path from either $(1,4)$ or $(6,4)$ to $(6,8)$, we have that $(5,8)$ is infected, meaning $(4,8)$ is uninfected. If $(2,8)$ is uninfected, there is now an uninfected path from either $(1,4)$ or $(6,4)$ to $(1,8)$. Therefore, $(2,8)$ is infected. Thus $(2,7)$ is an uninfected cell surrounded by four infected cells. Therefore, within the first eight columns, there is necessarily an uninfected cell surrounded by four infected cells or a second pair of adjacent infected cells within, as desired.

\end{proof}

We now begin our proof of Theorem~\ref{thm:rectall}. The next several theorems together find a matching upper bound to the lower bound values whenever $m,n\ge 7$. We introduce a few general tools and then resolve the remaining cases of Theorem~\ref{thm:rectall} by dividing into cases dependent on the values of $m,n\pmod{6}$.

We begin with a tool inspired by the proof of Theorem 1.8 in \cite{DNR23} that we will make use of in several cases. 

\begin{figure}[H]
    \centering
    \includegraphics[scale=0.6]{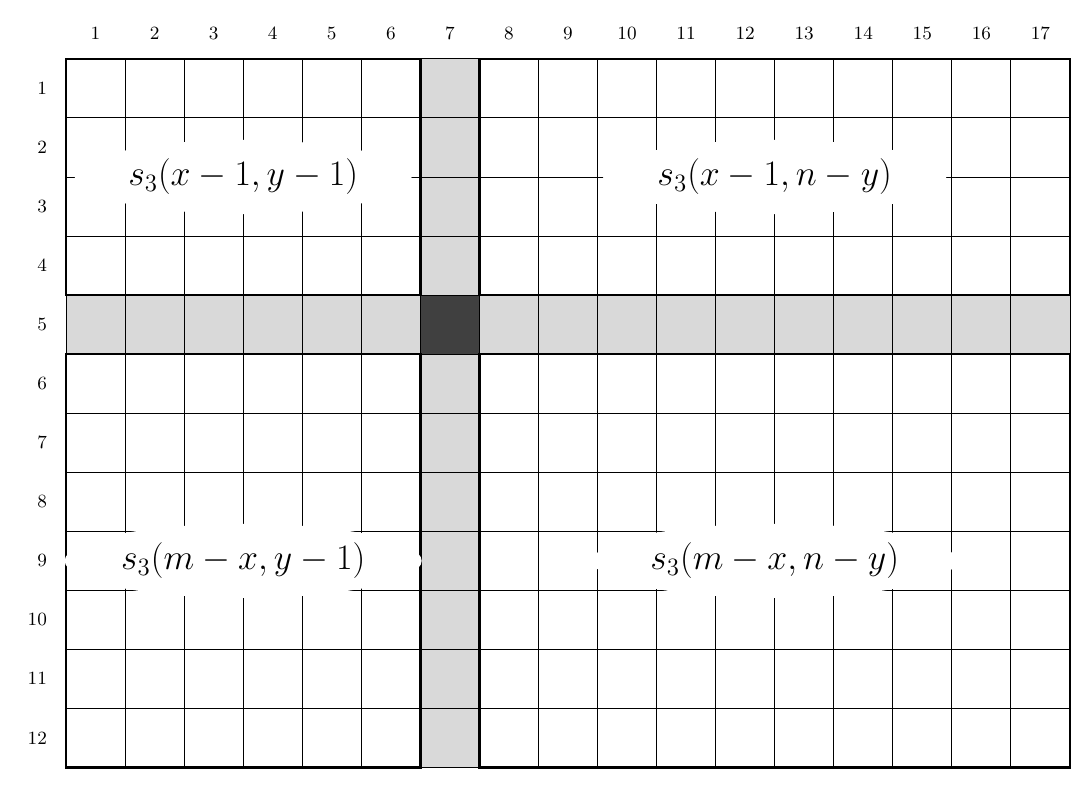}
    \caption{The Four Rectangle Trick, Theorem \ref{4rt}.}\label{fig:4rect}
\end{figure}

\begin{theorem}[Four Rectangle Trick]\label{4rt}
    For $2\le x\le m-1, 2\le y \le n-1$,
    \[
    s_3(m,n)\le 1+s_3(x-1,y-1)+s_3(m-x,y-1)+s_3(x-1,n-y)+s_3(m-x,n-y).
    \]
\end{theorem}
\begin{proof}
    As illustrated in Figure \ref{fig:4rect}, consider the following intersection pattern. In the upper left $(x-1)\times (y-1)$ rectangle, infect $s_3(x-1,y-1)$ cells in a way that percolates. In the lower left $(m-x)\times(y-1)$ rectangle, infect $s_3(m-x,y-1)$ cells in a way that percolates. In the upper right $(x-1)\times(n-y)$ rectangle, infect $s_3(x-1,n-y)$ cells in a way that percolates. In the lower right $(m-x)\times(n-y)$ rectangle, infect $s_3(m-x,n-y)$ cells in a way that percolates. Lastly infect $(x,y)$. This is a set of $1+s_3(x-1,y-1)+s_3(m-x,y-1)+s_3(x-1,n-y)+s_3(m-x,n-y)$ infected cells.

    Treating the four corner rectangles separately, every cell outside of row $x$ and column $y$ becomes infected. Additionally, since $(x,y)$ is already infected, the remaining uninfected cells form four disjoint paths. In each, one endpoint of the path is a neighbor of $(x,y)$ and thus becomes infected. Now the next cell along the path has three infected neighbors and becomes infected. For each of the four paths, the process cascades until reaching the boundary of the rectangle, at which point the entire grid is infected.
\end{proof}

\begin{corollary}\label{cor:rect3}
    For $m,n\ge 7$, $s_3(m,n)=\frac{mn+m+n}{3}+1$ when $m\equiv n\equiv 3\pmod{6}$ but either $m\neq n$ or $m+1$ is not a power of $2$.
\end{corollary}
\begin{proof}
    Since $m,n$ are both odd, $LB=\frac{mn+m+n}{3}$, but by Theorem~\ref{rectangleC0}, this bound is only tight when $m=n=2^k-1$ for some integer $k$. This establishes a lower bound of $\frac{mn+m+n}{3}+1$. For a corresponding upper bound, we apply Theorem~\ref{4rt} with $x=y=4$. Then,
    \begin{align*}
        s_3(m,n)\le 1+s_3(3,3)+s_3(m-4,3)+s_3(3,n-4)+s_3(m-4,n-4).
    \end{align*}

    Note from Theorem~\ref{rectangleC0} that $s_3(3,3)=5$. Both $m-4,n-4$ are $5\pmod{6}$, and from Theorem~\ref{thm:oldcases}, $s_3(a,b)=\frac{ab+a+b+1}{3}$ when $a,b\equiv 5\pmod{6}$. Thus,
    \begin{align*}
        s_3(m,n)&\le 1+5+\frac{(m-3)(4)}{3}+\frac{4(n-3)}{3}+\frac{(m-3)(n-3)}{3}\\
        &=6+\frac{(m+1)(n+1)-16}{3}=\frac{mn+m+n}{3}+1.\\
    \end{align*}
\end{proof}

For even $n\ge 8$, let $B$ be the following intersection pattern for a $6\times n$ rectangular grid. In rows 1 and 5, the cells in odd columns are infected; for rows 2 and 4, the cells in even columns are infected; in row 3, only the first cell is infected; in row 6, only the last cell is infected. For an example of $B$ with $n=12$, see Figure~\ref{figure:B}.

\begin{figure}[H]
    \centering
    \includegraphics[scale=0.6]{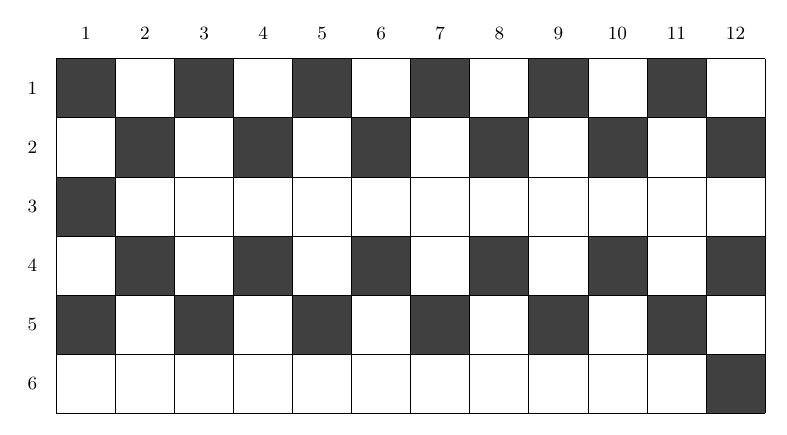}
    \caption{The block $B$ from Lemma \ref{lem:repeatB}.}\label{figure:B}
\end{figure}

\begin{lemma}\label{lem:repeatB}
    Consider a $6b\times n$ grid where the infection pattern is $b$ copies of $B$ stacked on top of each other, except that the last cell in the first row is also infected. Then every cell outside of the last row becomes infected.
\end{lemma}
\begin{proof}
    Within each copy of $B$, every uninfected cell in rows 1,2,4,5 except the last cell in row 1 initially has three infected neighbors so becomes infected. Now that rows 2 and 4 are entirely infected, cells in row 3 get infected one column at a time, cascading from column 2 to column $n$. For every copy of $B$ except the top one (where it is already infected), the last cell in the top row has three infected neighbors including the cell above it, so it becomes infected. Now that the top row of each copy of $B$ is infected, consider the bottom row of any copy except the last one. As the entire row below it is infected, one cell gets infected at a time, cascading from column $n-1$ down to column $1$.
\end{proof}

\begin{theorem}\label{thm:rectangle4even}
Suppose $m,n\ge 7$ and $m\equiv 4\pmod{6}$.
    If $n\equiv 0\pmod{6}$, then
    \[
    s_3(m,n)\le\frac{mn+m+n+5}{3}.
    \]
    If $n\equiv 4\pmod{6}$, then
     \[
    s_3(m,n)\le\frac{mn+m+n}{3}+2.
    \]
\end{theorem}
\begin{proof}
To establish the upper bounds, consider the following infection pattern. In both cases, the first $m-4$ rows will consist of $(m-4)/6$ copies of $B$ stacked on top of each other, with the last cell in the first row additionally infected. In row $m-3$, all odd columns are infected. In row $m-2$, the cells in columns $0,2\pmod{6}$ are infected. In row $m-1$, the first cell is infected along with everything in columns $3,5\pmod{6}$. In row $m$, the first cell is infected along with all even columns. Additionally, when $n\equiv 4\pmod{6}$, the last cell in row $m-2$ is infected. See Figure~\ref{figure:40} for an example when $m=10, n=12$. 

\begin{figure}[H]
    \centering
    \includegraphics[scale=0.6]{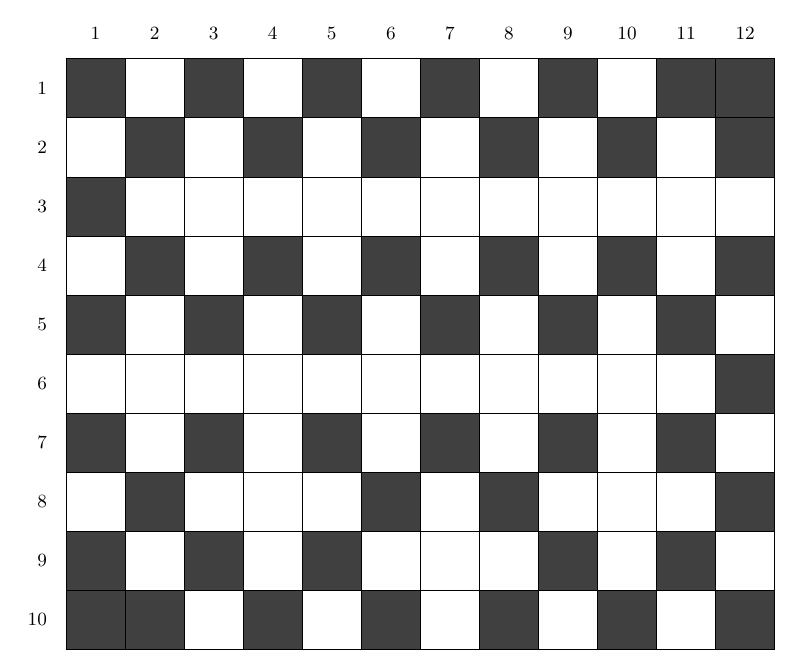}
    \caption{A minimum percolating set in the $10\times12$ grid, as per Theorem~\ref{thm:rectangle4even}.}\label{figure:40}
\end{figure}

For $n\equiv 0\pmod{6}$, the total number of infected cells is 
\[
(\frac{m-4}{6})(2n+2)+1+n/2+n/3+(n/3+1)+(n/2+1)=\frac{mn+m+n+5}{3}.
\]

For $n\equiv 4\pmod{6}$, the total number of infected cells is 
\[
(\frac{m-4}{6})(2n+2)+1+n/2+(n+2)/3+(n+2)/3+(n/2+1)=\frac{mn+m+n+6}{3}.
\]
It suffices to check that every cell becomes infected. From Lemma \ref{lem:repeatB}, we know that this is the case for the first $m-5$ rows.

In both cases, the last cell in the row $m-3$ has three infected neighbors so becomes infected. We will then show that the entire last four rows become infected. This will then be enough to infect the cells in row $m-4$ one at a time, cascading from column $n-1$ down to column 1.

In row $m-2$, any cell in columns $3,5\pmod{6}$ gets infected since the cells above and below are already infected while the cell to the left is infected for columns $3\pmod{6}$ and the cell to the right is infected for columns $5\pmod{6}$. In row $m-1$, all the cells in even columns become infected. For columns $2\pmod{6}$, this is from cells above, below, and to the right being infected. For columns $4\pmod{6}$, this is from columns to left, right, and below being infected (except for the last column with $n\equiv 4\pmod{6}$ where the cells above, below, and to the left are infected). For columns $0\pmod{6}$, this is from columns above, below, and to the left being infected.

Now every cell in the row $m$ in columns $3,5\pmod{6}$ and every cell in row $m-3$ in columns $2\pmod{6}$ has three infected neighbors and becomes infected. The uninfected cells in columns $0\pmod{6}$ of row $m-3$ (which when $n\equiv 0\pmod{6}$ do not include the last column as this is already infected) have infected cells to the left, right, and below, so become infected.
For columns $1\pmod{6}$, the cell in row $m-2$ has three infected neighbors and becomes infected. This is enough to infect the remaining uninfected cells below it, cascading down (except in the first column where there are no such remaining cells). This leaves just the cells in rows $m-3,m-2$ in columns $4\pmod{6}$. Such uninfected cells in row $m-2$ have infected neighbors to the left, right, and below, so become infected. Then such uninfected cells in row $m-3$ have infected neighbors to the left, right, and below, so become infected, completing the proof.

\end{proof}

For odd $n\ge5$, let $A$ be the following infection pattern for a $6\times n$ rectangular grid: in rows $1,3$, all odd cells are infected except column $3$; in row $2$, only columns $2,4$ are infected; in rows $4,6$, all even cells are infected; in row $5$, just the first and last cells are infected. For odd $n\ge 9$, let $A'$ be the following infection pattern for a $6\times n$ rectangular grid: in rows $1,3$, all odd cells are infected except column $7$; in row $2$, only columns $6,8$ are infected; in rows $4,6$, all even cells are infected; in row $5$, just the first and last cells are infected. For examples of $A$ and $A'$ with $n=11$, see Figures~\ref{figure:A} and \ref{figure:A'}, respectively. 
\begin{figure}[H]
    \centering
    \includegraphics[scale=0.6]{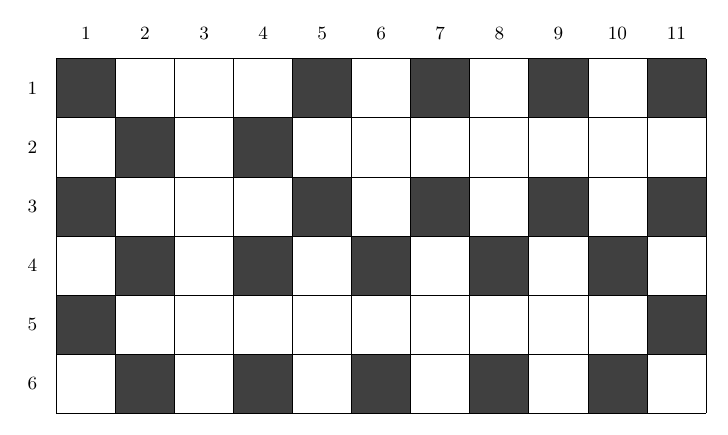}
    \caption{The repeating block $A$ from Lemma~\ref{lem:AorA'}.}\label{figure:A}
\end{figure}

\begin{figure}[H]
    \centering
    \includegraphics[scale=0.6]{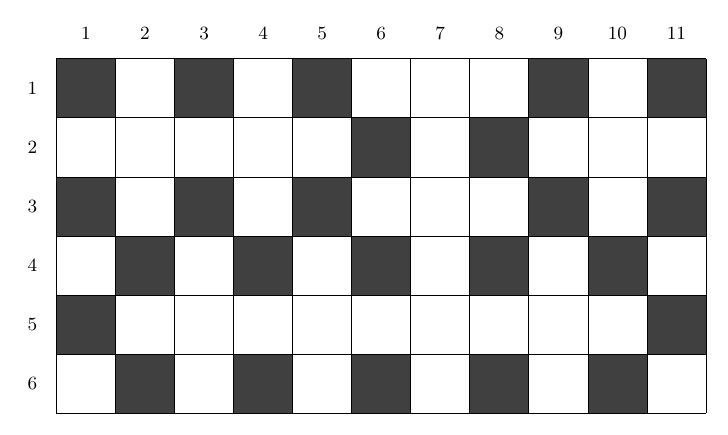}
    \caption{The repeating block $A'$  from Lemma~\ref{lem:AorA'}.}\label{figure:A'}
\end{figure}

\begin{lemma}\label{lem:AorA'}
    For $b\ge 1$, consider a $(6b+1)\times n$ grid where the infection pattern is $b$ copies of $A$ stacked on top of each other, along with precisely the odd cells in the bottom row infected. Then, every cell outside of the first two rows will become infected. The same holds for $A'$ instead of $A$.
\end{lemma}
\begin{proof}
    Within each copy of $A$, every uninfected cell in rows $3,4$ except for column $3$ has at least three infected neighbors and thus becomes infected. Within each copy of $A$, the row below row $6$ (either row $1$ of the copy of $A$ below, or the last row of the entire grid) has all its odd cells infected. This is enough for row $6$ to become infected except for column $3$.
    With the exception of the top copy of $A$, row $1$ of $A$ is a below a row (row $6$ of the copy of $A$ above) that has all cells infected except column $3$. This is enough for row $1$ to become infected except for column $3$.
    Note that since all even columns in row $6$ of the last copy of $A$ are infected, the entire bottom row of the whole grid becomes infected. Additionally the cell in row $5$, column $2$ and the cell in row $2$, column $1$ of each copy of $A$ each have three infected neighbors and become infected.

    At this point, in the bottom three rows of each copy of $A$, the only uninfected cells are in column $3$ and in columns $4$ to $n-1$ of row $5$. Beginning with the cell in row $5$, column $n-1$ and cascading to the left, everything in row $5$ gets infected except column $3$.
    For any copy of $A$ besides the top one, in the top three rows, the only uninfected cells are in column $3$ and in columns $5$ to $n$ of column $2$. Beginning with the cell in row $2$, column $5$ and cascading to the right, everything in row $2$ gets infected except column $3$.

    Now with the exception of some cells in the top two rows of the entire grid, everything is infected except column $3$. Since $(m,3)$ is already infected, beginning from the cell in row $6$, column $3$ of the bottom copy of $A$ and cascading upwards, the entire column $3$ gets infected except for the top row of the entire grid.

By an analogous argument, the same conclusion holds for $A'$ in place of $A$.
    
\end{proof}

In each of the following proofs of Theorems~\ref{thm:rect01} to~\ref{thm:rect43}, there is one value for the number of rows $m$ which corresponds to having no copies of the block $A$ or $A'$. In these cases, the proof does not work exactly as written and will be addressed in the remark after Theorem~\ref{thm:rect43}. While $A$ and $A'$ retain the same meaning throughout, the individual proofs will make use of a repeating block $C$ which is redefined within each proof. Additionally, we will at times refer to a set of disjoint paths. By this, we mean a collection of vertex-disjoint paths with no edges between different paths (this is slightly stronger than just requiring the paths to not share cells).

\begin{theorem}\label{thm:rect01}
    If $m,n\ge 7$, $m\equiv 0\pmod{6}$, and $n\equiv 1\pmod{6}$, then \[
    s_3(m,n)\le\frac{mn+m+n+2}{3}.
    \]
\end{theorem}
\begin{proof}
    To establish the upper bound for $n\ge 13$, consider the following infection pattern. The infected cells in the first eleven rows are as follows: all odds; cell 2 and all cells $0,4\pmod{6}$; all cells $1,3\pmod{6}$ but cell $3$; cells $1,3,n$ and all cells $2\pmod{3}$ but $2$; cell $2$ and all cells $1\pmod{3}$ but cells $1,n$; cells $1,n-2,n$ and all cells $0\pmod{3}$ but $3,n-1$; cells $2,n-1$ and all cells $1\pmod{3}$ except $1,n-3,n$; cells $1,n-2,n$ and all cells $0\pmod{3}$ but $n-1$; all cells $1\pmod{3}$ except cells $1,n$; cells $1,n$ and all cells $3,5\pmod{6}$; all cells $0,2\pmod{6}$. The next $m-12$ rows consist of $(m-12)/6$ copies of $A'$ stacked on top of each other. Finally, in the last row all odd cells are infected. For an example when $m=24,n=19$ see Figure \ref{fig:s01}.

\begin{figure}[H]
    \centering
    \includegraphics[scale=0.6]{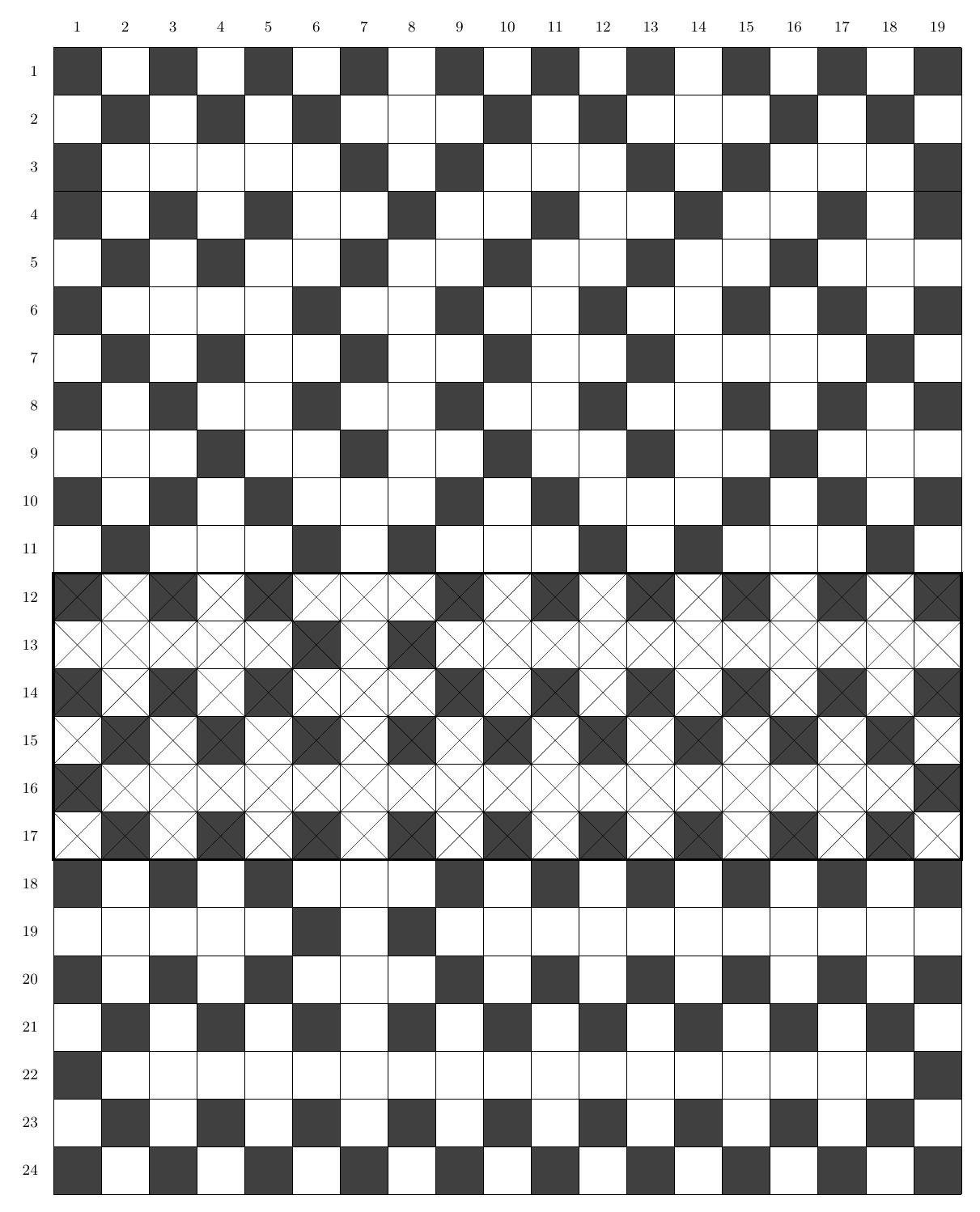}
    \caption{A minimum percolating set in the \mbox{$24\times19$} grid, as per Theorem~\ref{thm:rect01}, with repeating block $A'$ marked.}\label{fig:s01}
\end{figure}
    
    This gives a total of \begin{align*}&(n+1)/2+(n+2)/3+(n-1)/3+(n+5)/3+(n-1)/3+(n+2)/3+(n-1)/3\\&+(n+5)/3+ (n-4)/3+(n+5)/3+(n-1)/3+\frac{m-12}{6}(4(n-1)/2+2(2))+(n+1)/2\\
&=13n/3+14/3+(m/6-2)(2n+2)\\
&=\frac{mn+m+n+2}{3}\end{align*}
infected cells, so it suffices to check that every cell becomes infected. From Lemma~\ref{lem:AorA'}, we know that this is the case for all but the first $13$ rows (the case $m=12$ is handled in Remark~\ref{rmk:green}). Restricting our attention to the first $13$ rows, it is helpful to view columns $9$ through $n-5$ as $(n-13)/6$ copies of a repeating $13\times 6$ block which we call $C$.

Each copy of $C$ has some internal infections -- new cells which are infected by the infected cells in their own copy of $C$. Then, note that regardless of whether a copy of the block $C$ is to the right of another copy of block $C$ or whether it is the first copy and hence to the right of column $8$, the cell in row $11$ of the first column of $C$ becomes infected, consequently infecting the cells in row $11$, then row $12$ of the second column of $C$. Regardless of whether a copy of the block $C$ is to the left of another copy of block $C$ or whether it is the last copy and hence to the left of column $n-4$, the cell in row $12$ of the sixth column of $C$ becomes infected. Additionally in the first column of each copy of $C$ as well as in column $n-4$, rows $1$ and $3$ are initially infected and row $2$ becomes infected by its neighbors above, below, and to the right. Thus, regardless of whether a copy of block $C$ is to the left of another copy of block $C$ or whether it is the last copy and hence to the left of column $n-4$, the cell in row $3$ in the sixth column of $C$ becomes infected and consequently the cells in rows $2$, then $1$ of the sixth column of $C$ do as well. Ultimately, the uninfected cells within a copy of $C$ consist of two disjoint paths: a path from row $5$ in the first column to $C$ to row $5$ in the sixth column of $C$, and a horizontal path consisting of the entire row $13$ of $C$. Note in particular that for each copy of $C$, the entire row $13$ is infected.

Because the endpoint in row $5$ in the sixth column of a block $C$ is adjacent to the endpoint in row $5$ in the first column of the next block $C$, these paths can be joined together. Aside from row $13$, there are no other adjacencies between uninfected paths in different copies of $C$, so ultimately these form a single path of uninfected cells from $(5,9)$ to $(5,n-5)$. The only other uninfected cells in the range from column $9$ through column $n-5$ are those in row $13$.

Keeping in mind that the entire row $14$ is infected and that regardless of whether there are any copies of the block $C$, column $9$ is the same, every cell in the first eight columns becomes infected. For $n=13$, this establishes that every cell in the first $n-5$ columns becomes infected. We now show that this still holds for larger $n$. Because $(5,8)$ is now infected, this infects $(5,9)$ and the infection cascades through the entire path until $(5,n-5)$. Because $(13,8)$ is infected along with the entirety of rows $12$ and $14$ for columns $9$ through $n-5$, this causes $(13,9)$ to be infected, after which the infection cascades across row $13$ through column $n-5$. Thus we have established that every cell in the first $n-5$ columns becomes infected.

Before considering that uninfected cells in column $n-4$ can have infected neighbors in column $n-5$, we note that every cell in the last five rows becomes infected except for row $13$, $(12, n-3), (11, n-4), (11, n-3)$, and a path from $(5, n-4)$ to $(5,n)$ which is disjoint from these. As $(5, n-5)$ is infected, $(5, n-4)$ becomes infected by its neighbors to the left, right, and below, and this infection cascades along the path to $(5, n)$. $(11, n-4)$ becomes infected by virtue of its neighbors above, below, and to the left, and consequently infects $(11, n-3)$, followed by $(12, n-3)$. $(13, n-4)$ becomes infected by its neighbors above, below, and to the left, and this infection cascades across row $13$ through column $n$. Thus, the whole grid becomes infected.

For $n=7$, we use a different infection pattern. In rows $1$ and $3$, all odd cells are infected. In row $2$, cells $2$ and $6$ are infected. In row $4$, cells $1$ and $7$ are infected. In row $5$, cells $2,4,6$ are infected. The next $m-6$ rows consist of $(m-6)/6$ copies of $A$ stacked on top of each other. Finally, in the last row, all odd cells are infected. See Figure~\ref{figure:07} for an example when $m=12$.

This gives a total of
\[
2(4)+2+2+3+\left(\frac{m-6}{6}\right)(4(3)+2(2))+4=19+16\left(\frac{m-6}{6}\right)=\frac{8m+9}{3}=\frac{mn+m+n+2}{3}
\] infected cells, so it suffices to check that every cell becomes infected. From Lemma~\ref{lem:AorA'}, we know that this is the case for all but the first $7$ rows. Restricting our attention to the first $8$ rows, where we know the entire row $8$ is infected, it is easy to verify that every cell in the first $7$ rows becomes infected.

\begin{figure}[h]
    \centering
    \includegraphics[scale=0.6]{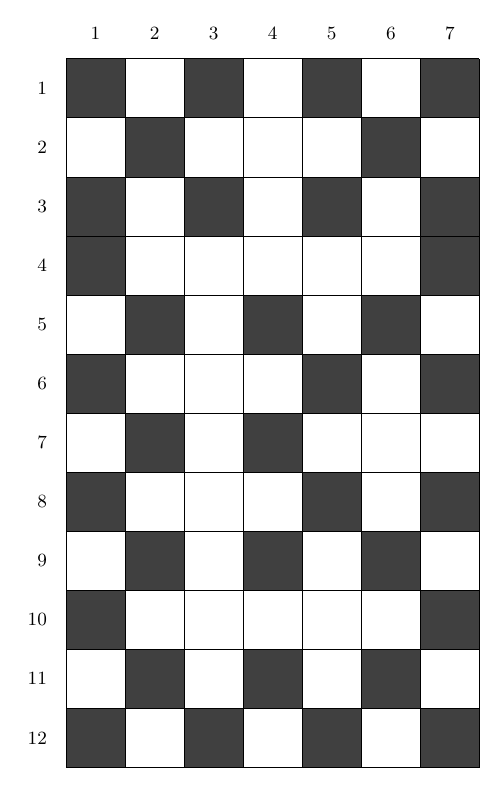}
    \caption{The $n=7$ case of Theorem \ref{thm:rect01}.}\label{figure:07}
\end{figure}

\end{proof}

\begin{figure}[H]
    \centering
    \includegraphics[scale=0.6]{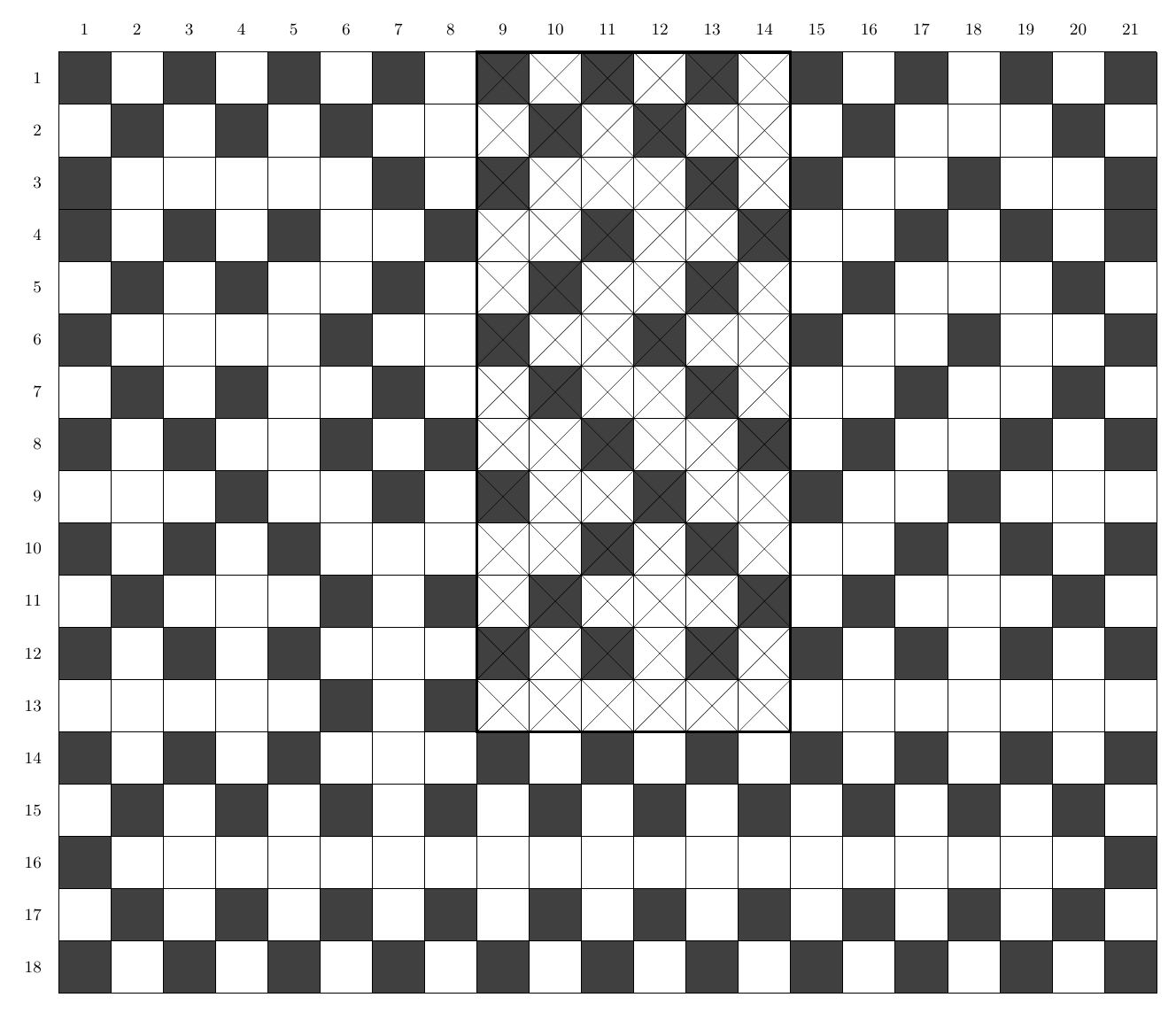}
    \caption{A minimum percolating set in the \mbox{$12\times 21$} grid, as per Theorem \ref{thm:rect03}; the block C is
marked with diagonal lines.}\label{figure:03}
\end{figure}

\begin{theorem}\label{thm:rect03}
    If $m,n\ge 7$, $m\equiv 0\pmod{6}$, and $n\equiv 3\pmod{6}$, then \[
    s_3(m,n)\le\frac{mn+m+n}{3}+1.
    \]
\end{theorem}
\begin{proof}

To establish the upper bound for $n\ge 15$, consider the following infection pattern. The infected cells in the first eleven rows are as follows: all odds; cells $2, n-1$ and all cells $0,4\pmod{6}$ but $n-3$; cell $n-3$ and all cells $1,3\pmod{6}$ but $3,n-2$; cells $1,3,n-2,n$ and all cells $2\pmod{3}$ but $2,n-1$; cells $2,n-1$ and all cells $1\pmod{3}$ but $1,n-2$; cell $1$ and all cells $0\pmod{3}$ but $3$; cells $2,n-4,n-1$ and all cells $1\pmod{3}$ but $1,n-5,n-2$; cells $1,3,6,n-5,n-2,n$ and all cells $2\pmod{3}$ but $2,5,n-4,n-1$; cells $4,7$ and all cells $0\pmod{3}$ but $3,6,n$; cells $3,n$ and all cells $1,5\pmod{6}$ but $7$; cell $6$ and all cells $2,4\pmod{6}$ but $4$. The next $m-12$ rows consist of $(m-12)/6$ copies of $A'$ stacked on top of each other. Finally in the last row, all cells in odd columns are infected. See Figure~\ref{figure:03} for the example of a $18\times 21$ grid. This gives a total of
\begin{align*}
   &(n+1)/2+n/3+n/3+(n/3+2)+n/3+n/3+n/3+(n/3+2)\\&+(n/3-1)+(n/3+1)+n/3+(m/6-2)(4(n-1)/2+4)+(n+1)/2\\
   &=13n/3+5+(m/6-2)(2n+2)\\
   &=13n/3+5+(mn/3-4n+m/3-4)\\
   &=\frac{mn+m+n}{3}+1
\end{align*} infected cells, so it suffices to check that every cell becomes infected. From Lemma~\ref{lem:AorA'}, we know that this is the case for all but the first $13$ rows (the case $m=12$ is handled in Remark \ref{rmk:green}). Restricting our attention to the first $13$ rows, it is helpful to view columns $9$ through $n-7$ as $(n-15)/6$ copies of a repeating $13\times 6$ block which we call $C$. 

Note that regardless of whether a copy of the block $C$ is to the right of another copy of block $C$ or whether it is the first copy and hence to the right of column $8$, the cell in row $11$ in the first column of $C$ becomes infected. In the first column of any block $C$ as well as in column $n-6$, the cells in rows $1$ and $3$ are initially infected and the cell in row $2$ becomes immediately infected by its neighbors above, below, and to the right. Thus, regardless of whether a copy of block $C$ is to the left of another copy of block $C$ or whether it is the last copy and hence to the left of column $n-6$, the cell in row $3$ in the sixth column of $C$ becomes infected, and consequently the cells in rows $2$ and $1$ of the sixth column of $C$ do as well. Ultimately, the uninfected cells within a copy of $C$ consist of four disjoint paths: a path from row $10$ in the first column of $C$ to row $7$ in the first column of $C$; a path from row $5$ in the first column of $C$ to row $10$ in the sixth column of $C$; a path from row $5$ in the sixth column of $C$ to row $7$ in the sixth column of $C$; and a horizontal path consisting of the entire row $13$ of $C$. Note in particular that in every copy of $C$, the entire row $12$ is infected.

Keeping in mind that the entire row $14$ is infected and that regardless of whether or not there are any copies of block $C$, column $9$ is the same, every cell in the first eight columns becomes infected besides $(5,8), (6,8)$, and $(7,8)$. Regardless of whether there are any copies of block $C$, $(4, n-7), (8,n-7), (11, n-7)$ are initially infected, so every cell in the last seven columns becomes infected except the following three disjoint paths: a path from $(5, n-6)$ to $(1, n-3)$, a path from $(7, n-6)$ to $(10, n-6)$, and a horizontal path from $(13, n-6)$ to $(13, n)$. Note in particular that the entire row $12$ is infected.

All cells in rows $12$ and $14$ are infected, along with $(13, 8)$, so $(13, 9)$ becomes infected, and the infection cascades until the entire row $13$ is infected.

Suppose there is at least one copy of $C$. For the uninfected paths among the copies of $C$, the endpoint in row $7$ of the first column of a block $C$ is adjacent to the endpoint in row $7$ of the sixth column of the previous block $C$ and the endpoint in row $5$ in the first column of a block $C$ is adjacent to the endpoint in row $5$ of the sixth column of the previous block $C$ (or in the case of the first copy of $C$, these endpoints are adjacent to uninfected cells in column $8$ which can be viewed as endpoints of an uninfected path from $(5,8)$ to $(7,8)$). Furthermore, the endpoint in row $10$ of the sixth column of a copy of $C$ is adjacent to the endpoint in row $10$ of the first column of the next copy of $C$. Thus, there is a single path of uninfected cells from $(10, 9)$ to $(10, n-7)$. This path then extends through the last $7$ columns, from $(10, n-6)$ to $(7, n-6)$. It then extends through $(7, n-7)$ to $(5, n-7)$, using the last remaining uninfected cells in the first $n-7$ columns. Finally it extends from $(5, n-6)$ to $(1, n-3)$. Thus, all remaining uninfected cells form a single path. The endpoint $(10, 9)$ becomes infected from its neighbors above, below, and to the left, and then the infection cascades through the entire path, until the whole grid is infected.

If there are no copies of $C$, there are exactly $15$ columns. The remaining uninfected cells form a single path from $(10, 9)$ to $(1, 12)$. $(10, 9)$ becomes infected from its neighbors above, below, and to the left, and then the infection cascades through the entire path, until the whole grid is infected.

For $n=9$, we use a different infection pattern. In rows $1$ and $3$, the cells in odd columns are infected. In rows $2$ and $5$, the cells in columns $2,6,8$ are infected. In row $4$, columns $1,4,9$ are infected. The rows from $6$ through $m-1$ consist of $(m-6)/6$ copies of $A'$ stacked on top of each other. Finally, in row $m$, the odd cells are infected. See Figure~\ref{figure:09} for an example with $m=12$. This gives a total of \[
2(5)+2(3)+3+(m/6-1)(4(4)+2(2))+5=24+20(m/6-1)=\frac{10m+12}{3}=\frac{mn+m+n}{3}+1
\] infected cells, so it suffices to check that every cell becomes infected. From Lemma~\ref{lem:AorA'}, we know that this is the case for all but the first $7$ rows. Restricting our attention to the first $8$ rows, where we know the entire row $8$ is infected, it is easy to verify that every cell in the first $7$ rows becomes infected.

\end{proof}

\begin{figure}[H]
    \centering
    \includegraphics[scale=0.6]{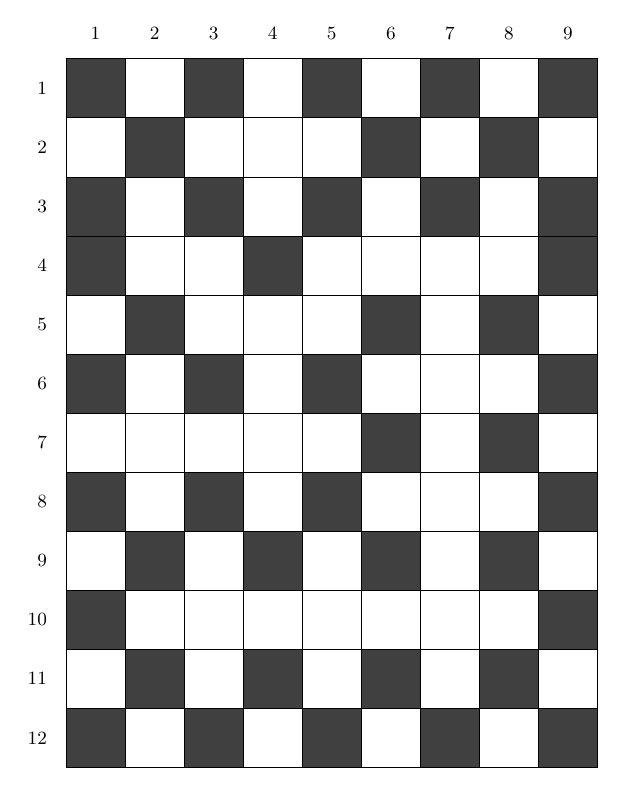}
    \caption{The $n=9$ case of Theorem \ref{thm:rect03}.}\label{figure:09}
\end{figure}

\begin{figure}[H]
    \centering
    \includegraphics[scale=0.6]{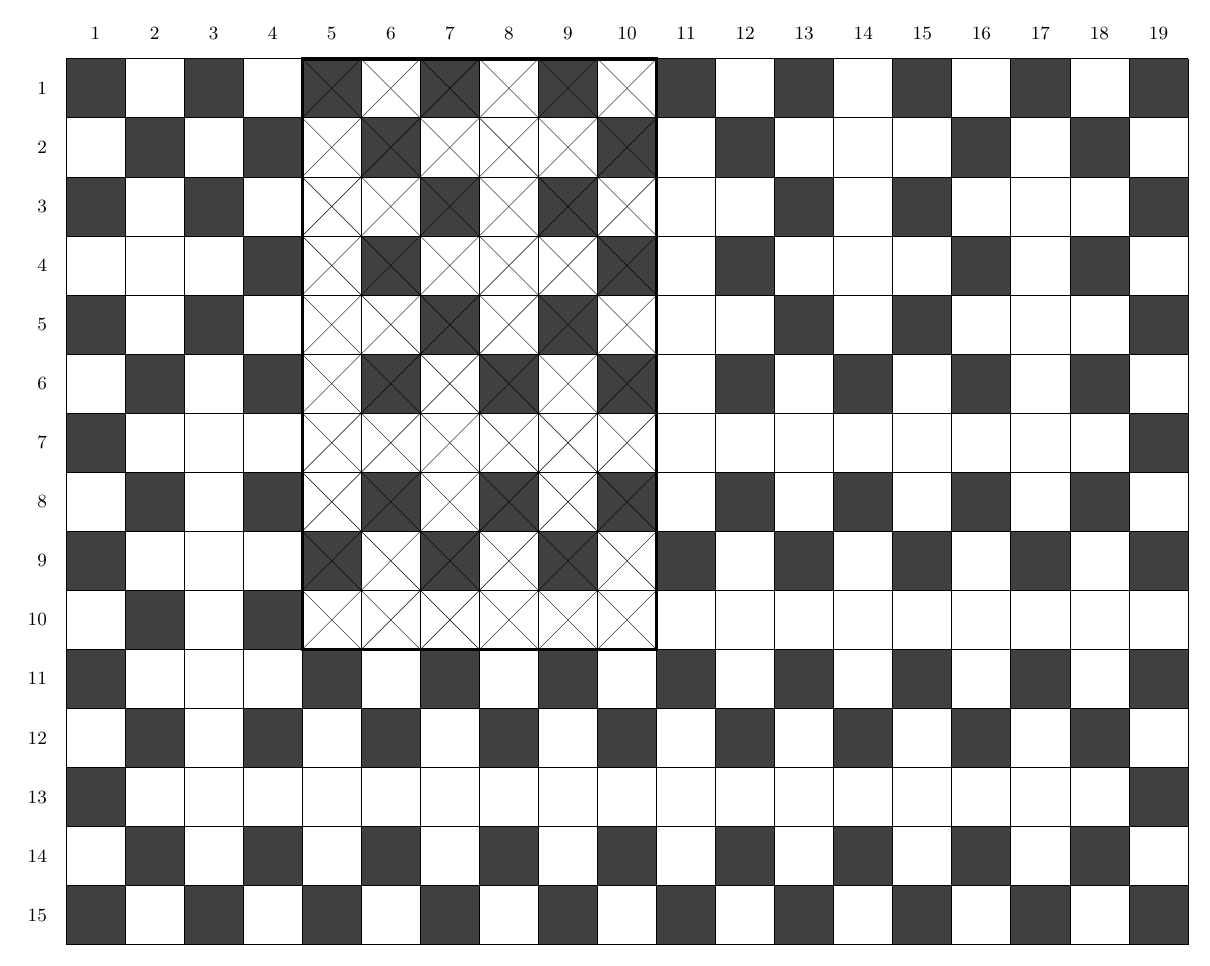}
    \caption{A minimum percolating set in the \mbox{$15\times 19$} grid, as per Theorem \ref{thm:rect13}; the block $C$ is marked with diagonal lines.}\label{figure:13}
\end{figure}

\begin{theorem}\label{thm:rect13}
    If $m,n\ge 7$, $m\equiv 3\pmod{6}$, and $n\equiv 1\pmod{6}$, then
    \[
    s_3(m,n)\le\frac{mn+m+n+2}{3}
.    \]
\end{theorem}
\begin{proof}
 To establish the upper bound, we use the following infection pattern. The infected cells in the first eight rows are as follows: all odds; cell $2$ and all cells $0,4\pmod{6}$; all cells $1,3\pmod{6}$; all cells $0,4\mod{6}$; all cells $1,3\pmod{6}$; all evens; just cells $1$ and $n$; all evens. The next $m-9$ rows consist of $(m-9)/6$ copies of $A$. 
Finally in the last row, the odd cells are infected. See Figure~\ref{figure:13} for the example of a $15\times 19$ grid. This gives a total of \begin{align*}&(n+1)/2+(n+2)/3+(n+2)/3+(n-1)/3+(n+2)/3\\&+(n-1)/2+2+(n-1)/2+(\frac{m-9}{6})(4(n-1)/2+2(2))+(n+1)/2\\&=\frac{mn+m+n+2}{3}\end{align*} infected cells, so it suffices to check that every cell becomes infected. From Lemma~\ref{lem:AorA'}, we know that this is the case for all but the first $10$ rows (the case $m=9$ is handled in Remark~\ref{rmk:green}). Restricting our attention to the first $10$ rows, it is helpful to view columns $5$ through $n-3$ as $(n-7)/6$ copies of a repeating $10\times 6$ block which we call $C$.

Note that regardless of whether a copy of the block $C$ is to the left of another copy of block $C$ or whether it is the last copy and hence to the left of column $n-2$, the cells in rows $1$ and $9$ in the sixth column of $C$ become infected. Also, in the sixth column of each copy of $C$ as well as in column $4$, the cells in rows $2,4,6,8$ are infected and those in rows $3,5$ immediately become infected. Consequently, the cells in rows $2,3,4,5,6,8$ of the first column of each copy of $C$ become infected. Ultimately, the only uninfected cells within a copy of $C$ are the entire row $7$ and the entire row $10$.

Regardless of whether there are any copies of $C$, $(1, 5)$ is initially infected so $(1, 4)$ becomes infected. Keeping in mind that the entire row $11$ is infected and that $(9, 5)$ is infected, $(9, 4)$ becomes infected, in turn infecting rows $(9, 3), (8, 3)$, and $(7, 3)$ in order, followed by $(7, 4)$. Thus every cell in the first $4$ columns is infected.

Since the entire row $11$, the entire row $9$ for columns $5$ through $m-3$, and $(10, 4)$ are infected, $(10, 5)$ becomes infected and the infection then cascades across row $10$ through column $m-3$. Similarly since every cell in rows $6$ and $8$ is infected for columns $5$ through $m-3$ and $(7, 4)$ is infected as well, $(7, 5)$ becomes infected and the infection then cascades across row $7$ through column $m-3$. 

Now the only remaining uninfected cells are in the last three columns. It is easy to see that every cell in the last two columns as well as $(10, m-2)$ becomes infected. Then $(2, m-2)$ becomes infected by virtue of its neighbors to the left, right, and above, and the infection cascades down the column through row $8$, leaving everything infected.
\end{proof}

\begin{figure}[H]
    \centering
    \includegraphics[scale=0.6]{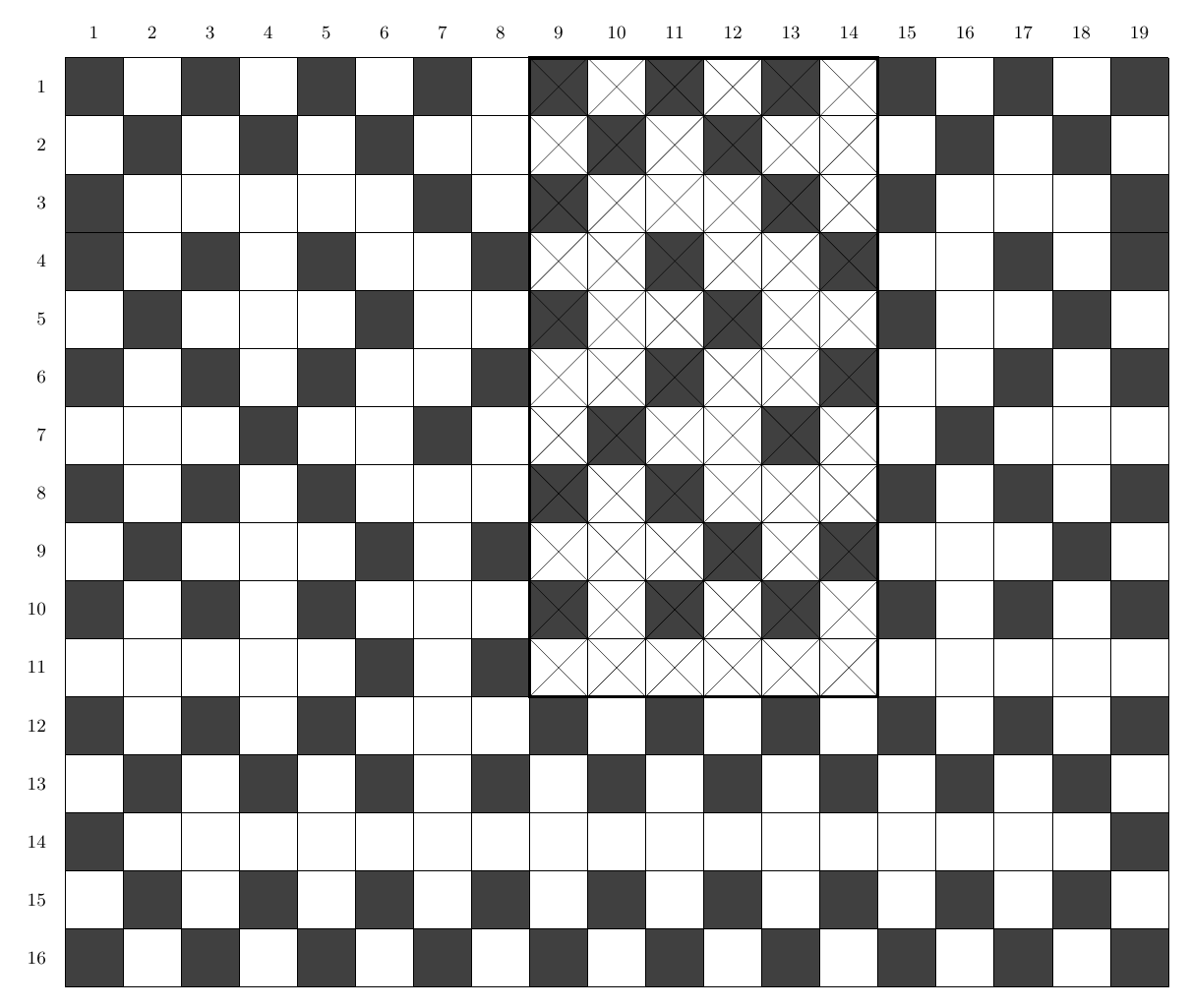}
    \caption{A minimum percolating set in the \mbox{$16\times 19$} grid, as per Theorem \ref{thm:rect44}; the block $C$ is marked with diagonal lines.}\label{figure:41}
\end{figure}

\begin{theorem}\label{thm:rect44}
    If $m,n\ge 7$, $m\equiv4\pmod{6}$, and $n\equiv 1\pmod{6}$, then 
    \[
    s_3(m,n)\le\frac{mn+m+n}{3}+1.
    \]
\end{theorem}
\begin{proof}
    To establish an upper bound for $n\ge 13$, the infected cells in the first nine rows are as follows: all odds; cell 2 and all cells $0,4\pmod{6}$; all cells $1,3\pmod{6}$ but cell $3$; cells $1,3,n$ and all cells $2\pmod{3}$ but $2$; cell $2$ and all cells $0\pmod{3}$ but $3$; cells $1,3,n$ and all cells $2\pmod{3}$ but $2$; all cells $1\pmod{3}$ but $1,n$; cells $1,n$ and all cells $3,5\pmod{6}$; all cells $0,2\pmod{6}$. The next $m-10$ rows consist of $(m-10)/6$ copies of $A'$. Finally in the last row, all odd cells are infected. See Figure~\ref{figure:41} for the example of a $16\times 19$ grid. This gives a total of
\begin{align*}
    &(n+1)/2+(n+2)/3+(n-1)/3+(n+5)/3+(n-1)/3+(n+5)/3\\&+(n-4)/3+(n+5)/3+(n-1)/3+(\frac{m-10}{6})(4(n-1)/2+4)+(m+1)/2\\
    &=11n/3+13/3+(\frac{m-10}{6})(2n+2)\\
    &=\frac{mn+m+n}{3}+1
\end{align*}
infected cells, so it suffices to check that every cell becomes infected. From Lemma~\ref{lem:AorA'}, we know that this is the case for all but the first $11$ rows (the case $m=10$ is handled in Remark~\ref{rmk:green}). Restricting our attention to the first $11$ rows, it is helpful to view columns $9$ through $n-5$ as $(n-13)/6$ copies of a repeating $11\times 6$ block which we call $C$.

Note that regardless of whether a copy of the block $C$ is to the right of another copy of block $C$ or whether it is the first copy and hence to the right of column $8$, the cells in rows $4$ and $9$ of the first column of $C$ become infected. Consequently, the cells in rows $9$ and $10$ in the second column of $C$ become infected. Regardless of whether a copy of block $C$ is to the left of another copy of $C$ or whether it is the last copy and hence to the left of column $n-4$, the cells in rows $5$ and $10$ in the sixth column of $C$ become infected. Additionally, in the first column of any block $C$ as well as in column $n-4$, the cells in rows $1$ and $3$ are initially infected and the cell in row $2$ becomes immediately infected by its neighbors above, below, and to the right. Thus for any copy of $C$, the cell in row $3$ of its sixth column becomes infected, and consequently the cells in rows $2$ and $1$ of that column do as well. Ultimately, the uninfected cells within a copy of $C$ form two disjoint paths: a path from row $7$ in the first column of $C$ to row $7$ in the sixth column of $C$ and a horizontal path consisting of the entire row $11$ of $C$. Note in particular that for each copy of $C$, the entire row $10$ is infected.

Because the endpoint in row $7$ in the sixth column of a block $C$ is adjacent to the endpoint in row $7$ of the first column of the next block $C$, these paths can be joined together, ultimately giving a path of uninfected cells from $(7, 9)$ to $(7, n-5)$. The only other uninfected cells in the range from column $9$ through column $n-5$ are those in row $11$.

Keeping in mind that the entire row $12$ is infected and that regardless of whether or not there are any copies of block $C$, column $9$ is the same, every cell in the first eight columns becomes infected. For $n=13$, this establishes that every cell in the first $n-5$ columns becomes infected. We now show that this still holds for larger $n$. If there are any copies of $C$, the fact that $(7, 8)$ is infected causes the endpoint in $(7, 9)$ to become infected and the infection to cascade down the path through $(7, n-5)$. Because $(11, 8)$ is infected along with the entirety of rows $10$ and $12$ for columns $9$ through $n-5$, this causes $(11, 9)$ to be infected, after which the infection cascades across row $11$ through column $n-5$. Thus, we have established that every cell in the first $n-5$ columns becomes infected.

Before considering  the fact that some cells in column $n-4$ have infected neighbors in column $n-5$, there are just $14$ cells in the last five columns which remain uninfected. These are the cells in row $11$ along with rows $4,6,7,9$ of column $n-4$ and rows $4,5,6,9,10$ of column $n-3$. As the entire column $n-5$ is already infected, this immediately leads to rows $4,7,9,11$ of column $n-4$ becoming infected. Consequently, $(6, n-4)$ becomes infected, which in turn infects $(6, n-3), (5, n-3), (4, n-3)$ in order. $(9, n-3)$ gets infected by virtue of its neighbors to the left, right, and above, which in turn infects $(10, n-3)$ and $(11, n-3)$. Then $(11, n-2)$ becomes infected, and the infection cascades to the end of the row.

\begin{figure}[H]
    \centering
    \includegraphics[scale=0.6]{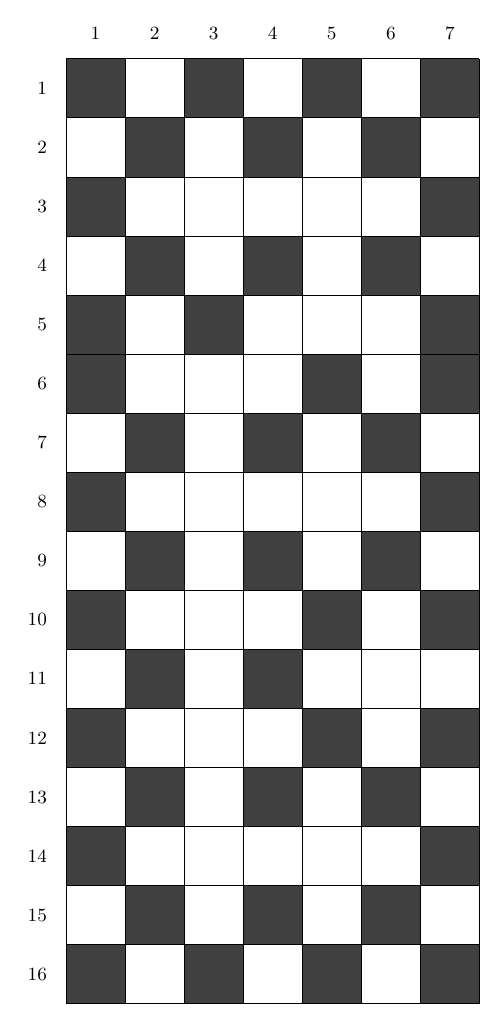}
    \caption{The $n=7$ case of Theorem \ref{thm:rect44}.}\label{figure:214b}
\end{figure}

If $n=7$, we use a different infection pattern. For row $1$, all odd columns are infected. For rows $2,4,7,9$, all even columns are infected. For rows $3$ and $8$, just columns $1,7$ are infected. For row $5$, columns $1,3,7$ are infected, while for row $6$, columns $1,5,7$ are infected. Rows $10$ through $m-1$ consist of $(m-10)/6$ copies of $A$ stacked on top of each other. Finally in the last row, all odd cells are infected. This gives a total of 
\[
2(4)+3(3)+2(2)+2+3+\left(\frac{m-10}{6}\right)(4(3)+2(2))+4=30+16\left(\frac{m-10}{6}\right)=\frac{8m+10}{3}=\frac{mn+m+n}{3}+1
\] infected cells, so it suffices to check that every cell becomes infected. For $m\ge 16$, then from Lemma~\ref{lem:AorA'}, we know that this is the case for all but the first $11$ rows. Restricting our attention to the first $12$ rows, where we know that the entire row $12$ is infected, it is easy to verify that every cell in the first $11$ rows becomes infected. For $m=10$, there are no copies of $A$ with which to apply Lemma~\ref{lem:AorA'}, but it is easy to verify that every cell becomes infected.
\end{proof}
\begin{theorem}\label{thm:rect43}
    If $m,n\ge 7$, $m\equiv 4\pmod{6}$, and $n\equiv 3\pmod{6}$, then\[
    s_3(m,n)\le \frac{mn+m+n+2}{3}.
    \]
\end{theorem}

\begin{figure}[H]
    \centering
    \includegraphics[scale=0.6]{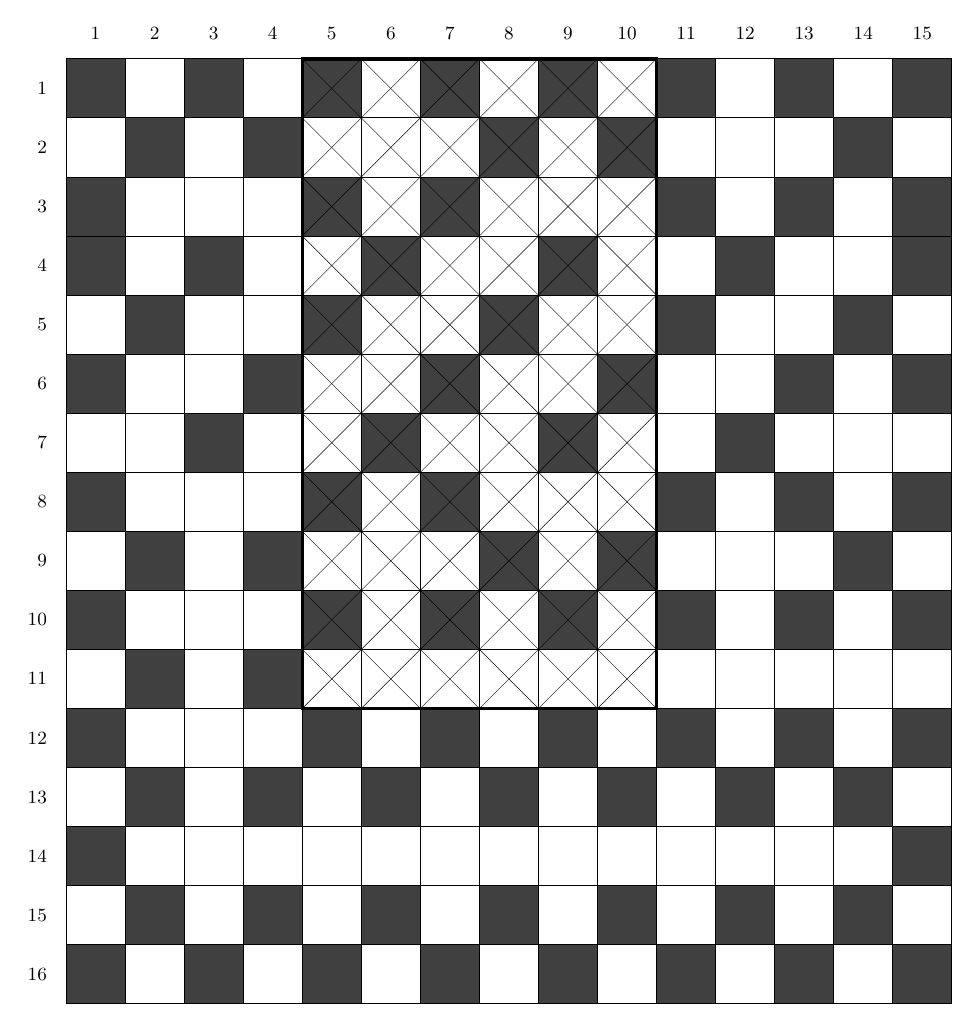}
    \caption{A minimum percolating set in the \mbox{$16\times 15$} grid, as per Theorem \ref{thm:rect43}; the block $C$ is marked with diagonal lines.}\label{figure:43}
\end{figure}

\begin{proof}
  To establish the upper bound, the infected cells in the first nine rows are as follows: all odds; all cells $2,4\pmod{6}$; cell $n$ and all cells $1,5\pmod{6}$; cell $1$ and all cells $0\pmod{3}$; all cells $2\pmod{3}$; cell $n$ and all cells $1\pmod{3}$; all cells $0\pmod{3}$ but $n$; cell $n$ and all cells  $1,5\pmod{6}$; all cells $2,4\pmod{6}$. The next $m-10$ rows consist of $(m-10)/6$ copies of $A$. Finally, in the last row, the odds cells are infected. See Figure~\ref{figure:43} for the example of a $16\times 15$ grid. This gives a total of \begin{align*}&(n+1)/2+n/3+2(n/3+1)+n/3+(n/3+1)+(n/3-1)+(n/3+1)+n/3\\&+\frac{m-10}{6}(4(n-1)/2+2(2))+(n+1)/2\\
&=11n/3+4+\frac{m-10}{6}(2n+2)=\frac{mn+m+n+2}{3}\end{align*} infected cells, so it suffices to check that every cell becomes infected. From Lemma~\ref{lem:AorA'}, we know that this is the case for all but the first $11$ rows (the case $m=10$ is handled in Remark~\ref{rmk:green}). Restricting our attention to the first $11$ rows, it is helpful to view columns $5$ through $n-5$ as $(n-9)/6$ copies of a repeating $11\times 6$ block which we call $C$.

Note that regardless of whether a copy of the block $C$ is to the right of another copy of block $C$ or whether it is the first copy and hence to the right of column $4$, the cells in rows $2$ and $9$ of the first column of $C$ become infected. Similarly, regardless of whether a copy of the block $C$ is to the left of another copy of $C$ or whether it is the last copy and hence to the left of column $n-4$, the cells in rows $1$ and $10$ of the sixth column of $C$ become infected. Ultimately, the uninfected cells within a copy of $C$ form two disjoint paths: a path from row $7$ in the first column of $C$ to row $7$ in the sixth column of $C$ and a horizontal path consisting of the entire row $11$ of $C$. Note that in particular for each copy of $C$, the entire row $10$ is infected.

Because the endpoint in row $7$ in the sixth column of a block $C$ is adjacent to the endpoint in row $7$ of the first column of the next block $C$, these paths can be joined together, ultimately giving a path of uninfected cells from $(7, 5)$ to $(7, n-5)$. The only other uninfected cells in the range from column $5$ through column $n-5$ are those in row $11$.

Keeping in mind that the entire row $12$ is infected and that regardless of whether or not there are any copies of block $C$, column $5$ is the same, every cell in the first four columns becomes infected besides a path from $(7, 1)$ to $(7, 4)$. If there are any copies of $C$, the endpoint in $(7, 4)$ is adjacent to the endpoint in $(7, 5)$, so these paths can be merged to a single path of uninfected cells from $(7, 1)$ to $(7, n-5)$.

Regardless of whether there are any copies of $C$, $(11,5)$ becomes infected. Because the entire row $12$ is infected, along with row $10$ for columns $5$ through $n-5$, the infection cascades across row $11$ through column $n-5$. Additionally, $(2, n-4), (6, n-4)$, and $(9, n-4)$ become infected. Consequently, rows $1,2,9,10$ of column $n-3$ become infected and $(7, n-4)$ becomes infected by virtue of its neighbors above, below, and to the right. The only remaining uninfected cells in the last five columns are in row $11$ and get infected by virtue of $(11, n-5)$ and the entire rows $10$ and $12$ being infected.

The remaining uninfected cells now from a single path from $(7, 1)$ to $(7, n-5)$. The endpoint $(7, n-5)$ gets infected by virtue of its neighbors to the left, right, and above, and the infection then cascades throughout the path.
\end{proof}

\begin{remark}\label{rmk:green} We now deal with the leftover small cases from the preceding five proofs where the number of rows is a particular small value, $m^*$ (e.g., $m^*=12$ for Theorem~\ref{thm:rect01}) such that there are no copies of the blocks $A$ and $A'$. In this case, we cannot perform the same reduction using Lemma~\ref{lem:AorA'} because as written, it would reduce the problem to $m^*+1$ rows, more than the original number. However, the stated construction with the same proof implies the desired result for $m^*$ rows.

When $m=m^*$, all the odd cells in row $m^*$ are initially infected. In each of the five relevant constructions (one from each of Theorems \ref{thm:rect01} - \ref{thm:rect43}), one can check that this is sufficient to infect row $m^*-1$ within two steps. As the first $m^*-1$ rows  were shown to become infected in the written proof, they will still become infected when row $m^*-1$ is replaced by a completely infected row. Once the first $m^*-1$ rows are infected, the remaining uninfected cells in the bottom row become infected by virtue of their neighbors to the left, right, and above, and thus every cell is infected, as desired.

\end{remark}

This completes the proof of \Cref{thm:rectall}.

\section{Open problems}
We completely resolved the $3$-neighbor bootstrap percolation problem on rectangular grids $P_m\square P_n$. However, it would be interesting to consider what happens when cells can be infected by their diagonal neighbors, that is when the graph is the strong product $P_m\boxtimes P_n$. 
\begin{question}
    What is the minimum number of infected cells needed to fully infect $P_m\boxtimes P_n$ under the $3$-neighbor bootstrap percolation process?
\end{question}
We quickly note that the answer is substantially lower than for $P_m\square P_n$, where roughly $1/3$ of the cells need to be infected. In fact, even for the $4$-neighbor bootstrap percolation process, it suffices to initially infect the boundary along with any cells with both coordinates odd, meaning roughly $1/4$ of the cells are infected for $m,n$ sufficiently large. Any remaining cell with both coordinates even instantly becomes infected by virtue of its four diagonal neighbors. Then any cell still uninfected becomes infected by its vertically and horizontally adjacent neighbors.

In light of \Cref{lem:trees}, the minimum size of a $3$-percolating set in the toroidal grid $C_m\square C_n$ for $m,n\ge 3$ is the same as its decycling number, i.e., the minimum number of vertices that can be removed in order to destroy all cycles. The decycling number for $C_m\square C_n$ was completely determined by Pike and Zou \cite{PZ05}. The $3$-neighbor bootstrap percolation problem remains open for cylindrical grids.

\begin{question}
    What is the minimum number of infected cells needed to fully infect $C_m\square P_n$ under the $3$-neighbor bootstrap percolation process?
\end{question}

It is also natural to consider rectangular grids in higher dimensions. In this setting, we let $s_3(x_1,x_2,\dots,x_d)$ denote the minimum number of initially infected cells needed to infect every cell of $P_{x_1}\square P_{x_2}\square\cdots\square P_{x_d}$. Dukes, Noel, and Romer~\cite{DNR23} considered $3$-dimensional grids and showed that \[s_3(x_1,x_2,x_3)=\lceil{\frac{x_1x_2+x_1x_3+x_2x_3}{3}\rceil}\] whenever $min\{x_1,x_2,x_3\}\ge 11$. This result was extended to $min\{x_1,x_2,x_3\}\ge 7$ by Dolphin and Dukes~\cite{DD25}. Recalling that having one dimension equal to $4$ or $6$ proved to be a nontrivial exceptional case in the two-dimensional setting (see Theorems~\ref{width4} and~\ref{width6}), we ask the following question. We avoid the case where any side has just a single cell, as this reduces to the lower dimensional case. 
\begin{question}
    Determine $s_3(x_1,x_2,x_3)$ for $2\le \min\{x_1,x_2,x_3\}\le 6$.
\end{question}

\section{Acknowledgments}
The authors thank Jaan Allikvere for directing them to \cite{PZ05}.


\end{document}